\documentclass[a4,11pt]{amsart}
\usepackage{amssymb,bm,amstext,amsmath,amscd,amsthm,
amsfonts,enumerate,graphicx,latexsym}
\usepackage[usenames]{color}
\usepackage[shortlabels]{enumitem}
\usepackage[all]{xy}

\usepackage{url}
\newtheorem{theorem}{Theorem}[section]
\newtheorem{corollary}[theorem]{Corollary}
\newtheorem{lemma}[theorem]{Lemma}
\newtheorem{proposition}[theorem]{Proposition}

\theoremstyle{definition}
\newtheorem{definition}[theorem]{Definition}
\newtheorem{remark}[theorem]{Remark}

\newtheorem*{question*}{Question}

\newtheorem*{conjecture*}{Conjecture}
\newtheorem{example}[theorem]{Example}

\newtheorem*{notation*}{Notation}

\newtheorem*{claim*}{Claim}

\newcounter{caseCount}

\newtheoremstyle{caseM}
  {5pt}
  {8pt}
  {}
  {}
  {\bfseries}
  {}
  {5pt plus 1pt minus 1pt}
  {\thmname{#1}~(\thmnumber{#2})~\underline{\thmnote{#3}}~:~}
\theoremstyle{caseM}
\newtheorem{caseM}[caseCount]{}
\numberwithin{equation}{theorem}
\def\Ext{\operatorname{Ext}}

\def\rad{\operatorname{rad}}
\def\soc{\operatorname{soc}}
\def\Hom{\operatorname{Hom}}

\def\End{\operatorname{End}}

\def\op{{\rm op}}

\def\Db{\mathsf{D}^{\rm b}}

\def\add{\operatorname{\mathsf{add}}}
\def\ind{\operatorname{\mathsf{ind}}}

\def\mod{\operatorname{\mathsf{mod}}}
\def\stmod{\operatorname{\underline{\mathsf{mod}}}}
\def\proj{\operatorname{\mathsf{proj}}}

\def\C{\mathcal{C}}

\def\Z{\mathbb{Z}}

\def\Triv{\mathrm{Triv}}
\def\Tria{\triangle}

\newcommand{\Module}[1]
{\mbox{$\begin{subarray}{c}
#1
\end{subarray}$}}
\newcommand{\mOdule}[2]
{\mbox{$\begin{subarray}{c}
#1\\
#2
\end{subarray}$}}
\newcommand{\moDule}[3]
{\mbox{$\begin{subarray}{c}
#1\\
#2\\
#3
\end{subarray}$}}
\newcommand{\modUle}[4]
{\mbox{$\begin{subarray}{c}
#1\\
#2\\
#3\\
#4
\end{subarray}$}}

\begin{document}
\setlength{\baselineskip}{15pt}
\title[Representation-finite gendo-symmetric algebras]{On representation-finite gendo-symmetric algebras with only one non-injective projective module}
\author{Takuma Aihara}
\address{Department of Mathematics, Tokyo Gakugei University, 4-1-1 Nukuikita-machi, Koganei, Tokyo 184-8501, Japan}
\email{aihara@u-gakugei.ac.jp}
\author{Aaron Chan}
\address{Graduate School of Mathematics, Nagoya University, Furocho, Chikusaku, Nagoya 464-8602, Japan}
\email{aaron.kychan@gmail.com}
\author{Takahiro Honma}
\address{Department of Mathematics, Tokyo University of science, 1-3 Kagurazaka, sinjukuku, Tokyo 162-8601, Japan}
\email{1119704@ed.tus.ac.jp}

\keywords{gendo-symmetric algebra, symmetirc algebra, representation-finite algebra}
\thanks{2010 {\em Mathematics Subject Classification.} 18E30, 16E35, 13D09}
\thanks{TA was partly supported by JSPS Grant-in-Aid for Young Scientists 19K14497.
AC is supported by Research Activity Start-up 19K23401.
}

\begin{abstract}
Motivated by the relation between Schur algebra and the group algebra of a symmetric group, along with other similar examples in algebraic Lie theory, Min Fang and Steffen Koenig \cite{FK1, FK2} addressed some behaviour of the endomorphism algebra of a generator over a symmetric algebra, which they called \emph{gendo-symmetric algebra}.  Continuing this line of works,  we classify in this article the representation-finite gendo-symmetric algebras that have at most one isomorphism class of indecomposable non-injective projective module.   We also determine their almost $\nu$-stable derived equivalence classes in the sense of Wei Hu and Changchang Xi \cite{HX}.  It turns out that a representative can be chosen as the quotient of a representation-finite symmetric algebra by the socle of a certain indecomposable projective module.
\end{abstract}
\maketitle
\section{Introduction}\label{sec:intro}

Taking the endomorphism ring is a fundamental construction in creating new algebras, and studying properties of the original algebra via the new ring is a central theme in representation theory.  For instance, the endomorphism algebra of a progenerator yields a Morita equivalent algebra, and the endomorphism algebra of a tilting module yields a derived equivalent algebra.  In particular, the existence of these equivalences allow one to go back-and-forth between these algebras and thus provide greater flexibility in study their representations.

Perhaps the most classical example where one recovers the original algebra $\Lambda$ from the endomorphism algebra $\Gamma=\End_\Lambda(M)$ is when the module $M$ has the \emph{double centraliser property}, i.e. when there is an algebra isomorphism
\[
\Lambda \cong \End_{\End_{\Lambda}(M)^{\op}}(M).
\]
In particular, $M$ is faithful as a $\Gamma^\op$-module.
The study of double centraliser property has a deep root in the algebraic Lie theory, namely, a fundamental example is given by the Schur-Weyl duality.  That is, when $M$ is the tensor space $(\mathbb{C}^n)^{\otimes n}$, $\Lambda$ is the group algebra $\mathbb{C}\mathfrak{S}_n$ of the symmetric group, which acts on $M$ via place permutation, and $\Gamma$ is the Schur algebra; see, for example, \cite{Gre} for details.

The \emph{Morita--Tachikawa correspondence} \cite{Mor,Tac} says that a if a $\Lambda$-module $M$ is projective as $\Gamma^\op$-module, then $M$ has the double centraliser property precisely when it is a \emph{generator-cogenerator} of the category $\mod \Lambda$ of finitely generated modules, i.e. when every indecomposable projective and indecomposable injective module appear as a direct summand of $M$.
It also gives an internal characterisation of the algebra $\Gamma$ by a certain homological condition. Apart from the Schur-Weyl duality example, another particular interesting examples arising in this theory is the \emph{Auslander algebras}, which appears as $\Gamma=\End_\Lambda(M)$ for a representation-finite algebra $\Lambda$ with $M$ the multiplicity-free module such that $\mod\Lambda=\add M$ \cite{Aus2}.

Motivated by examples of the double centraliser property in algebraic Lie theory, Fang and Koenig \cite{FK1} gave a refinement of the Morita--Tachikawa correspondence for the case when the base algebra $\Lambda$ is \emph{symmetric}, thus giving an internal characterisation of \emph{endo}morphism rings of a \emph{gen}erator(-cogenerator) over a symmetric algebra - or \emph{gendo-symmetric} for short \cite{FK2}.  Their studies further emphasise how closely gendo-symmetric algebras resemble various behaviour of symmetric algebras.  For related studies in this subject, see, for example, \cite{M1, M2, M3, CM, B}.

One beautiful result in the representation theory of symmetric algebras  is the derived (and stable) equivalence classification of the representation-finite symmetric algebras over an algebraically closed field, which is given by a certain extension of the Dynkin classification.  
\begin{definition}\label{def:RFSy}
Simplifying the notation in \cite{Asa}, we categorise the {\bf r}epresentation-{\bf f}inite {\bf sy}mmetric algebras by their \emph{RFSy-type}\footnote{We use RFSy to distinguish from the self-injective case, which is termed `RFS-type' in \cite{Asa}.} - a pair $(\Delta,f)$ in one of the following forms:
\begin{itemize}
\item $(\Delta,1)$ with $\Delta$ a simply-laced Dynkin graph (or quiver) - representing the class of trivial extension algebras of iterated tilted Dynkin types;
\item $(\mathbb{A}_{mn},1/m)$ with $m,n\in\Z_{>0}$ - representing the class of Brauer tree algebras;
\item $(\mathbb{D}_{3n},1/3)$ with $n \in \Z_{>1}$ - representing the class of modified Brauer tree algebra.
\end{itemize}
\end{definition}
Assuming for simplicity that the characteristic of the underlying field is not $2$, then every derived (or stable) equivalence class of representation-finite symmetric algebras is uniquely determined by an RFSy-type.  We refer to subsection \ref{subsec:RFS} for the details as well as the clarification in the characteristic $2$ case.

Being a generalisation of symmetric algebras, it is natural to ask if there is a similar classification of representation-finite \emph{gendo}-symmetric algebras.  A caveat in this problem is that, unlike symmetric algebras, derived equivalences need not preserve the gendo-symmetricity.  Fortunately, it is known from \cite{CM} that this problem can be resolved by restricting to \emph{almost $\nu$-stable} derived equivalences (in the sense of Hu and Xi \cite{HX}) instead; see subsection \ref{subsec:derived} for details.  Subsequently in the same paper, the second-named author and Marczinzik classified representation-finite \emph{biserial} gendo-symmetric algebra and their almost $\nu$-stable derived equivalence classes.  This time, we turn to another class of (non-symmetric) gendo-symmetric algebra - those that are the closest to being symmetric, i.e. of the form $\End_\Lambda(\Lambda\oplus M)$ with $M$ indecomposable non-projective.

To state the our result, recall first (see subsection \ref{subsec:RFS} for details) that if a representation-finite symmetric algebra $A$ is of RFSy-type $(\Delta,f)$, then for an indecomposable $A$-module, its $\tau$-orbit can be identified with a vertex, denoted by $v(M)$, of $\Delta$.  This vertex is well-defined up to the canonical non-trivial graph automorphism (if there is one) on $\Delta$.  On the other hand, for any  vertex $v$ of $\Delta$, one can define an \emph{extension graph} $\Delta^v$ which is given by adding an extra vertex to $\Delta$ and connecting it to $v$.
Now we can state the classification as follows.

\begin{theorem}\label{thm:main}
Suppose $B$ is a ring-indecomposable non-simple symmetric algebra over an algebraically closed field.
For any indecomposable non-projective $B$-module $M$, the following are equivalent.
\begin{enumerate}[\rm (1)]
\item $\End_B(B\oplus M)$ is a representation-finite gendo-symmetric algebra;

\item $B$ is representation-finite symmetric of type $(\Delta,f)$ and $v:=v(M)\in \Delta_0$ satisfies one of the following (mutually exclusive) conditions.
\begin{itemize}
\item $f=1$ and the extension graph $\Delta^v$ is Dynkin.
\item $(\Delta,f)=(\mathbb{A}_{mn},1/m)$ for some $m>1$, $n\geq 1$, and  $v$ is of valency 1.
\item $(\Delta,f)=(\mathbb{D}_{3r},1/3)$ for some $r>1$ and $v$ is the valency 1 vertex in the longest branch of the graph $\mathbb{D}_{3r}$.
\end{itemize}

\item There is a representation-finite symmetric algebra $C$ and an indecomposable projective module $P$ such that $\End_B(B\oplus M)$ is almost $\nu$-stable derived equivalent to $C/\soc P$.
\end{enumerate}
\end{theorem}

This article is structured as follows.  In Section \ref{sec:prelim}, we recall all basic facts required for our result.  In Section \ref{sec:trivext}, we will give some detail analysis on certain gendo-symmetric algebras with trivial extension base algebra.  In Section \ref{sec:modBTree}, we deal specifically with the case of modified Brauer tree.  In Section \ref{sec:final}, we will prove the main theorem and also give a few remarks on it.

\section{Preliminaries}\label{sec:prelim}

Throughout this paper, algebras are always assumed to be finite-dimensional over an algebraically closed field $K$.
By module we mean a finitely generated right module.
The $K$-linear dual is denoted by $D:=\Hom_K(-,K)$.
We assume, unless otherwise stated, that all of our algebras are ring-indecomposable, i.e. its $\Ext^1$-quiver is connected.
For simplicity, we also assume all algebras are basic.

For an algebra $\Lambda$, we denote by $\mod\Lambda$ the category of modules over $\Lambda$, by $\ind\Lambda$ the set of isomorphism classes of indecomposable $\Lambda$-modules.  Its bounded derived category is denoted by $\Db(\mod\Lambda)$.
The opposite algebra of $\Lambda$ is denoted by $\Lambda^\op$.

When $\Lambda$ is $\Z$-graded, we denote the category of $\Z$-graded modules by $\mod^\Z\Lambda$.  The stable category (i.e. the category with the same objects but Hom-spaces are given by quotienting out the morphisms that factor through a projective module) of $\mod\Lambda$ and of $\mod^\Z\Lambda$ are denoted by $\stmod\Lambda$ and $\stmod^\Z\Lambda$ respectively.

\subsection{Constructions of algebras}\label{ssec:prelim alg}
\begin{definition}
Let $\Lambda$ be an algebra.
\begin{enumerate}[(1)]
\item A $\Lambda$-module $M$ is a \emph{generator} if for any $X\in\mod\Lambda$, there is an epimorphism $M'\to X$ with $M'\in\add M$.
This is equivalent to having every indecomposable (isoclass of) projective module appears as a direct summand of $M$.

\item A \emph{gendo algebra over $\Lambda$} is a triple $(\Gamma, M, \Lambda)$ of algebras $\Gamma, \Lambda$ and generator $M\in \mod\Lambda$ such that $\Gamma \cong \End_\Lambda(M)$.  We often call the algebra $\Gamma$ here a gendo algebra for simplicity.  We also call $\Lambda$ the \emph{base algebra}.

\item An algebra $\Lambda$ is \emph{symmetric} if $\Lambda\cong D\Lambda$ as $\Lambda$-$\Lambda$-bimodule.  A \emph{gendo-symmetric algebra} is a gendo algebra whose base algebra is symmetric.
\end{enumerate}
\end{definition}

\begin{remark}
Every algebra can be realised as a gendo algebra via $(\Gamma,\Gamma,\Gamma)$.  More generally, $\Gamma$ is Morita equivalent to $\Lambda$ if and only if $(\Gamma,M,\Lambda)$ is a gendo algebra with $M$ a progenerator (i.e. $M\in \add\Lambda$ such that $\add M = \add \Lambda$). In particular, a gendo-symmetric algebra being symmetric is equivalent to $M$ being a progenerator.
\end{remark}

We recall in particular that for a symmetric algebra $\Lambda$, there is a natural isomorphism of the $\Lambda$-dual and the $K$-dual functors:
\begin{align}\label{eq:dual}
\Hom_\Lambda(-,\Lambda) \cong \Hom_K(-,K) : \mod\Lambda \xrightarrow{\sim} \mod\Lambda^\op
\end{align}

An important class of symmetric algebra comes from the following construction.
\begin{definition}
Let $\Lambda$ be an algebra.  The \emph{trivial extension of $\Lambda$}, denoted by $\Triv(\Lambda)$, is the space $\Lambda\oplus D\Lambda$ equipped with multiplication
\[
(a,f)(b,g) := (ab, ag+fb).
\]
\end{definition}

One can check that $D\Lambda$ is an ideal in $\Triv(\Lambda)$, and that $\Triv(\Lambda)$ is always symmetric via the isomorphism $\Triv(\Lambda)\xrightarrow{\sim} D\Triv(\Lambda)$ given by $(a,f)\mapsto[(b,g)\mapsto f(b)+g(a)]$.  Note that there is a canonical embedding functor $\mod \Lambda \to \mod \Triv(\Lambda)$ induced by the algebra homomorphism $\Triv(\Lambda) \twoheadrightarrow \Lambda \cong \Triv(\Lambda)/D\Lambda$.

\begin{definition}
Let $\Lambda$ and $\Gamma$ be algebras.
Take a $\Gamma$-$\Lambda$-bimodule ${}_\Gamma M_\Lambda$, a $\Lambda$-$\Gamma$-bimodule ${}_\Lambda N_\Gamma$, a $\Gamma$-$\Gamma$-bimodule homomorphism $f:M\otimes_\Lambda N\to \Gamma$, and a $\Lambda$-$\Lambda$-bimodule homomorphism $g:N\otimes_\Gamma M \to \Lambda$.  
Then these induce a multiplication map on the vector space
\[\begin{pmatrix} \Lambda & N \\ M& \Gamma \end{pmatrix} := \left\{\begin{pmatrix} a & n \\ m & b \end{pmatrix} \middle|\, a\in \Lambda, b\in \Gamma, m \in M, n\in N \right\},\]
where elements are multiplied like matrix multiplication.
This algebra is called the \emph{Morita context ring} associated to $({}_\Gamma M_\Lambda, {}_\Lambda N_\Gamma, f, g)$.
\end{definition}

\begin{example}\label{eg:Morita-context}
Let $\Lambda$ be an algebra.
\begin{enumerate}
\item A Morita context $(_{\Gamma}M_\Lambda,{}_\Lambda N_{\Gamma},f,g)$ defines a Morita equivalence between $\Lambda$ and $\Gamma$ if and only if $f$ and $g$ are surjective (or equivalently, any one of them is surjective).

\item Take any $M\in \mod \Lambda$.  Then the gendo algebra $\End_\Lambda(\Lambda\oplus M)$ is isomorphic to a Morita context ring
\[
\begin{pmatrix}
\Lambda & {}_\Lambda\Hom_\Lambda(M,\Lambda)_\Gamma \\ {}_\Gamma M_\Lambda & \Gamma
\end{pmatrix},
\]
where $\Gamma:=\End_\Lambda(M)$ and the defining bimodule homomorphisms are
\begin{align*}
\begin{array}{ccc}
M \otimes_\Lambda \Hom_\Lambda(M,\Lambda) & \to& \Gamma \\
m\otimes f & \mapsto & [m' \mapsto m f(m')]
\end{array} \text{ and }
\begin{array}{ccc}
\Hom_\Lambda(M,\Lambda)\otimes_\Gamma M & \to& \Lambda \\
f\otimes m & \mapsto& f(m)
\end{array}
\end{align*}
This is sometimes called \emph{Auslander context} \cite{Buc}.

\item The special case of a Morita context ring where the lower left corner $M$ is zero is called the (upper) \emph{triangular matrix ring} of a $\Lambda$-$\Gamma$-bimodule ${}_\Lambda N_\Gamma$, and will be denoted by $\Tria({}_\Lambda N_\Gamma)$.  If, furthermore, $\Gamma$ is the simple algebra $K$, then the associated triangular matrix ring is also called the \emph{one-point coextension algebra of $\Lambda$ by the left $\Lambda$-module ${}_\Lambda N$}.
\end{enumerate}
\end{example}

\subsection{Representation type of algebras}\label{ssec:prelim reptype}

From now on, we always count indecomposable objects up to isomorphism.

An idempotent-closed additive category $\C$ is \emph{of finite type} if it admits only finitely many indecomposable objects.  

The \emph{representation type} of an algebra is one of the following.
\begin{enumerate}
\item \emph{Representation-finite}: when $\mod\Lambda$ is of finite type.
\item \emph{Representation-tame}: when all but finitely many isomorphism classes of indecomposable modules of a given (finite) dimension belong to finitely many one-parametric families.
\item \emph{Representation-wild}: there are non-trivial two-parameter families of finite-dimensional indecomposable modules.
\end{enumerate}
For details on the definition, see the original script \cite{Dro} or the survey \cite{Sko}.

By \emph{representation-infinite} we mean that $\mod \Lambda$ is not representation-finite.  Clearly, being representation-finite implies representation-tame, so we will call a representation-infinite algebra of tame type \emph{infinite-tame} for convenience and for clarity.
By Drozd's dichotomoy theorem \cite{Dro}, tame and wild are mutually exclusive conditions.

Note that representation type of an algebra is left-right symmetric due to the duality $D:\mod \Lambda\xrightarrow{\sim} \mod\Lambda^\op$.
Clearly, as $\mod \Lambda/I$ and $\mod e\Lambda e$ embeds into $\mod \Lambda$ for any two-sided ideal $I$ and any idempotent $e$, an algebra with a representation-infinite quotient or with a representation-infinite idempotent subalgebra is also representation-infinite.

On the other hand, there is a special case when the representation-finiteness of a quotient implies that of the original algebra - this is given by the following Drozd-Kirichenko rejection lemma \cite{DK}.
\begin{lemma}[Rejection Lemma \cite{DK}]\label{lem:reject}
Let $\Lambda$ be an algebra and $P$ be an indecomposable projective-injective $\Lambda$-module.
Then the natural embedding $\mod \Lambda/\soc(P)\to \mod \Lambda$ induces a bijection between $\ind\Lambda/\soc(P)$ and $\ind\Lambda\setminus \{ P \}$.
In particular, $\Lambda/\soc(P)$ is representation-finite (resp. representation-tame) if and only if so is $\Lambda$.
\end{lemma}

A useful tool in determining representation-infiniteness is to use the separated quiver.

\begin{definition}
For a quiver $Q$, its associated \emph{separated quiver} is the quiver $Q^s$ defined as follows.
Suppose $Q_0 = \{1,2,\ldots, n\}$.  Then \[
Q^s_0:=\{1, 2, \ldots, n, 1', 2', \ldots, n'\}.
\]
Every arrow of $Q^s$ is of the form
$\xymatrix{i \ar[r] & j'}$
whenever there is an arrow $\xymatrix{i \ar[r] & j}$ in $Q$.

For an algebra $\Lambda$, its \emph{separated quiver} is the separated quiver associated to the $\Ext^1$-quiver of $\Lambda$.
\end{definition}

Here is a well-known result.

\begin{proposition}{\rm\cite[X. Theorem 2.6]{ARS}}\label{prop:RFRad}
An algebra whose radical squares to zero is representation-finite if and only if its separated quiver is a finite disjoint union of Dynkin quivers.
In particular, if the separated quiver of an algebra contains a connected component that is non-Dynkin, then the algebra is representation-infinite.
\end{proposition}

\begin{example}\label{Exrad2}
Consider quiver of the form $\xymatrix{\ar@(ul,dl)_{}1\ar@(ur,dr)^{}}$.
Then its separated quiver is the Kronecker quiver $\xymatrix{1 \ar@<0.5ex>[r] \ar@<-0.5ex>[r]&1'}$. Hence, if the $\Ext^1$-quiver of an algebra has two loops attached to a vertex, then the algebra is representation-infinite. 
\end{example}

\subsection{Representation-finite symmetric algebras} \label{subsec:RFS}
We start by recalling two fundamental results (Theorem \ref{tricot}, Theorem \ref{thm:RFSym}), which are consequences of a long list of works including \cite{G,GR,Rie,BLR,HW,Sko} and many more.
We also refer to the survey \cite[Section 3, 4]{Sko} for any unexplained definitions, such as iterated tilted algebras, Brauer tree algebras, and modified Brauer tree algebra; we will only use certain special cases of these algebras.

The first result concerns the trichotomy of trivial extension algebras of acyclic path algebras.

\begin{theorem}\label{tricot}
Let $\Delta$ be a finite acyclic quiver.
The following are equivalent.
\begin{enumerate}[\rm (i)]
\item $\Delta$ is a (simply-laced) Dynkin (resp. Euclidean, resp. wild) quiver;
\item $K\Delta$ is representation-finite (resp. of infinite-tame type, resp. of wild type);
\item $\Triv(K\Delta)$ is representation-finite (resp. of infinite-tame type, resp. of wild type).
\end{enumerate}
\end{theorem}

We also recall the classification of representation-finite symmetric algebras for completeness.

\begin{theorem}\label{thm:RFSym}
A non-simple basic symmetric algebra is representation-finite if and only if it is isomorphic to one of the following mutually exclusive classes of algebras.
\begin{enumerate}[\rm (i)]
\item Trivial extension of an iterated tilted algebra of Dynkin type.
\item Brauer tree algebra with exceptional multiplicity at least 2.
\item Modified Brauer tree algebra. If the char$K=2$, then this further splits into two classes, namely of parameter $\epsilon=0$ (standard type) and $\epsilon=1$ (non-standard type).
\end{enumerate}
\end{theorem}

We note that a modified Brauer tree algebra with parameter $\epsilon=1$ in the case when $\text{char} K\neq 2$ is isomorphic to one with parameter $\epsilon=0$; see \cite{W}.  We will give the definition of a special case in Section \ref{sec:modBTree}.

As is the case of the investigation carried out in this article, properties of representation-finite symmetric algebras can often be studied up to derived equivalence or stable equivalence, i.e. triangulated equivalences of the (bounded) derived categories $\Db(\mod\Lambda)$ or stable module categories $\stmod\Lambda$, respectively.  We refer to \cite{H} if one needs detail explanation on these terminologies.

We first recall a result of Rickard \cite{Ric} that relates derived and stable equivalences between self-injective algebras.
Suppose that $\Lambda$ is self-injective.  Then any derived equivalence $F:\Db(\mod \Lambda)\to \Db(\mod \Gamma)$ induce a stable equivalence $\overline{F}:\stmod\Lambda\to \stmod \Gamma$ and there is a commutative diagram
\[
\xymatrix{
\Db(\mod \Lambda) \ar[r]^{F}\ar[d]_{\pi_\Lambda} & \Db(\mod \Gamma)\ar[d]^{\pi_\Gamma} \\ \stmod \Lambda \ar[r]^{\overline{F}} & \stmod \Gamma,}
\]
where $\pi_\Lambda$ is the canonical functor $\Db(\mod \Lambda) \to \Db(\mod\Lambda)/\mathsf{K}^{\mathrm{b}}(\proj\Lambda) \simeq \stmod{\Lambda}$.

From now on, we will focus only on the symmetric algebras.
As mentioned in the introduction, derived equivalences of representation-finite symmetric algebras are (almost) classified by the RFSy-type $(\Delta,f)$; it tells us the shape of its stable Ausalnder-Reiten quiver (or \emph{stable AR-quiver} for short).
Namely, $A$ is of RFSy-type if its stable AR-quiver takes the form of the translation quiver $\mathbb{Z}\Delta/\langle \tau^{m_\Delta f}\rangle$, where $m_\Delta:=h_\Delta-1$ and $h_\Delta$ is the Coxeter number of type $\Delta$.

It will be convenience to pick the derived equivalence class representative for each RFSy-type.

\begin{theorem}{\rm\cite{Asa}}\label{thm:asa}
Let $B$ be a representation-finite symmetric algebra of RFSy-type $(\Delta,f)$.
\begin{enumerate}[\rm (i)]
\item If $f=1$, then $B$ is derived equivalent to $\Triv(k\vec{\Delta})$ for any orientation $\vec{\Delta}$ on $\Delta$.

\item If $(\Delta,f)=(\mathbb{A}_{mn},1/m)$ for some $n\geq 1$ and some $m>1$, then $B$ is derived equivalent to a symmetric Nakayama algebra with $n$ simples and Loewy length $mn+1$ (a.k.a. Brauer star algebra with $n$ simples multiplicity $m$).

\item If $(\Delta,f)=(\mathbb{D}_{3n}, 1/3)$ for some $n\geq 2$, then $B$ is derived equivalent to a modified Brauer star algebra with $n$ simples for some $n\geq 2$ (see Definition \ref{def:modBSt}).  If $\mathrm{char} K\neq 2$, then this uniquely determine a derived equivalence class; otherwise, there are two derived equivalence classes, distinguished by a parameter $\epsilon\in\{0,1\}$:
\begin{itemize}
\item $\epsilon=0$: standard modified Brauer tree algebra.
\item $\epsilon=1$: non-standard modified Brauer tree algebra. 
\end{itemize} 
\end{enumerate}
\end{theorem}

\subsection{Almost $\nu$-stable derived equivalences}\label{subsec:derived}

For an algebra $\Lambda$, we denote by $\nu$ the Nakayama functor of $\Lambda$, i.e. $\nu := -\otimes_\Lambda D\Lambda:\add\Lambda \to \add D\Lambda$.
We recall in the following the notion of almost $\nu$-stable derived equivalences introduced in \cite{HX}.  Although its definition is somewhat cumbersome and possibly unenlightening, we will see in Proposition \ref{prop:nuder} the central role it plays in our investigation.

\begin{definition}
For a tilting complex $T$ concentrated in non-positive degrees, say $T=(T^{-n} \to T^{-n+1} \to \cdots \to T^1\to T^0)$, it is said to be \emph{negatively $\nu$-stable} if $\add(\bigoplus_{i=1}^n T^{-i}) = \add(\bigoplus_{i=1}^{n}\nu(T^{-i}))$.  Dually we define \emph{positively $\nu$-stable} for tilting complexes concentrated in non-negative degrees.

Suppose $F:\Db(\mod\Lambda)\to\Db(\mod\Gamma)$ is a derived equivalence, and $F^{-1}$ is its quasi-inverse.
Up to applying sufficiently many shifts or swapping $F$ with $F^{-1}$, we say that $F$ (and also $F^{-1}$) is \emph{almost $\nu$-stable} if $F(\Lambda)$ and $F^{-1}(\Gamma)$ are positively $\nu$-stable and negatively $\nu$-stable respectively.
\end{definition}

Note that our formulation differs slightly from \cite{HX}; this is soley for the purpose of shortening the conditions and there is no difference from the original definition.
For us, the most important aspect can be summarised in the following slogan: `almost $\nu$-stable derived equivalences of gendo-symmetric algebras behave in the same way as derived equivalences of symmetric algebras'.
This can be justified by the following result.

\begin{proposition}\label{prop:nusta}
Let $F:\Db(\mod\Lambda)\to \Db(\mod\Gamma)$ be an almost $\nu$-stable derived equivalence. Then the following hold.
\begin{enumerate}
\item {\rm \cite[Theorem 3.7]{HX}} $F$ induces a stable equivalence between $\overline{F}:\stmod\Lambda\to \stmod\Gamma$.
\item {\rm \cite[Proposition 4.1, Corollary 5.4]{HX}} $\Lambda$ and $\Gamma$ have the same global, finitistic, representation, and dominant dimensions.
\item {\rm \cite[Proposition 6.1]{HX}} There is an induced almost $\nu$-stable derived equivalence between $\End_{\Lambda}(\Lambda\oplus M)$ and $\End_{\Gamma}(\Gamma\oplus \overline{F}(M))$, where $\overline{F}$ is the induced stable equivalent in (i).
\item {\rm \cite{Ric2}} $\Lambda$ is symmetric if and only if so is $\Gamma$.
\item {\rm \cite[Theorem 4.6]{CM}} $\Lambda$ is gendo-symmetric if and only if so is $\Gamma$.
\end{enumerate}
\end{proposition}


As a consequence, we have the following result, which will be used frequently in the sequel.

\begin{proposition}\label{prop:nuder}
Suppose $A$ and $B$ are symmetric algebras and $M$ is a $A$-module.
If $A$ is representation-finite and $\Phi:\stmod{A}\to \stmod{B}$ is a stable equivalence, then $\End_{A}(A\oplus M)$ is representation-finite if and only if so is $\End_{B}(B\oplus \Phi(M))$.
\end{proposition}
\begin{proof}
First recall from \cite{Asa2} (for the case of RFSy-type $(\Delta,1)$ and $(\mathbb{A}_{mn},1/m)$) and Dugas \cite{Dug} (for the case of RFSy-type $(\mathbb{D}_{3m},1/3)$) that $\Phi$ can always be lifted to a derived equivalence $F$, i.e. $\Phi=\overline{F}$ in the notation of Proposition \ref{prop:nusta} (1).

Since these algebras are symmetric, any derived equivalence between them is necessarily (almost) $\nu$-stable.  Hence, it follows from Proposition \ref{prop:nusta} (3) that 
$\End_{A}(A\oplus M)$ and $\End_{B}(B\oplus \Phi(M))$ are almost $\nu$-stable derived equivalence.
In particular, it follows from Proposition \ref{prop:nusta} (2) that they have the same representation dimension.
Since the representation dimension of an algebra is $2$ if and only if it is representation-finite \cite{Aus2}, the claim follows.
\end{proof}

The crucial point in the proof above is the preservation of representation dimension.  This argument also allows us to obtain a general form of the implication from (3) to (1) of the main Theorem \ref{thm:main}.

\begin{proposition}\label{prop:main(3)=>(1)}
Suppose $B$ is a symmetric algebra and $M\in \mod B$ is a generator.
If there is a representation-finite symmetric algebra $\Lambda$ and an ideal $I$ of $\Lambda$ such that $\End_B(M)$ is almost $\nu$-stable derived equivalent to $\Lambda/I$, then $\End_B(M)$ is representation-finite.
\end{proposition}
\begin{proof}
Since that representation-finiteness of $\Lambda$ implies that of $\Lambda /I$, it follows from Auslander's result \cite{Aus2} that the representation dimension of $\Lambda/I$ is $2$.  Hence, by Proposition \ref{prop:nusta} (2), the representation dimension of $\End_B(M)$ is also $2$, which in turn means that $\End_B(M)$ is also representation-finite.
\end{proof}

We expect that the converse holds when $I$ is a nilpotent ideal such that $eI=0$ where $e\in \Lambda$ is the idempotent such that $e(\Lambda/I)e\cong B$.  Note that the main Theorem \ref{thm:main} verifies this speculation in the case when $M$ consists of only one indecomposable non-projective direct summand.

\section{The case when the base algebra is a trivial extension}\label{sec:trivext}

As trivial extension provides a rather well-understood class of symmetric algebras, we look at gendo-symmetric algebras whose base is a trivial extension algebra.

\subsection{Description via Morita context ring} \label{ssec:triv-gendo}
In this section, let $\Lambda$ be an algebra and $M$ be a $\Lambda$-module.
Put $E:=\End_{\Lambda}(M)$.
Then we have $K$-linear maps
\begin{align}\label{eq:dualbimod}
\begin{array}{ccc}
DM\otimes_E M & \to& D\Lambda \\
f \otimes m & \mapsto & (a\mapsto (f(ma)))
\end{array} \text{ and } 
\begin{array}{ccc}
M\otimes_\Lambda DM & \to& DE \\
m\otimes f& \mapsto & (\theta \mapsto f(\theta(m)))
\end{array}.
\end{align}
It is routine to check that these maps are $\Lambda$-$\Lambda$-bimodule homomorphism and an $E$-$E$-bimodule homomorphism, respectively.  Since $M$ is an $E$-$\Lambda$-bimodule, we can replace $\otimes_\Lambda$ by $\otimes_{\Triv(\Lambda)}$, and $\otimes_E$ by $\otimes_{\Triv(\Lambda)}$, respectively in these maps.  Hence, composing with the inclusions of bimodules $D\Lambda\to \Triv(\Lambda)$ and $DE\to \Triv(E)$ respectively, we obtain bimodule homomorphisms
\[
DM\otimes_{\Triv(E)} M \to \Triv(\Lambda) \quad\text{and}\quad  M\otimes_{\Triv(\Lambda)} DM \to \Triv(E).
\]
Therefore, we have a Morita context ring
\[\begin{pmatrix}
\Triv(\Lambda) & DM \\ M & \Triv(E)
\end{pmatrix} \text{ with a 2-sided ideal }\begin{pmatrix}
0 & 0 \\ 0 & DE\end{pmatrix}.\]
The induced quotient ring is the Morita context ring
\[\begin{pmatrix}
\Triv(\Lambda) & DM \\ M & E
\end{pmatrix}\]
where the defining bimodule homomorphism $DM\otimes_{\Triv(E)}M\to \Triv(\Lambda)$ is the same as before while $M\otimes_{\Triv(\Lambda)} DM \to 
\Triv(E)$ becomes zero.

The following algebra isomorphisms will be very helpful in understanding a gendo-symmetric algebra whose base is a trivial extension.

\begin{lemma}\label{lem:triv(tria)}
There are algebra isomorphisms
\begin{align}
\Triv(\Tria({}_{\Lambda}DM_E)) & \cong \begin{pmatrix}
\Triv(\Lambda) & DM \\ M & \Triv(E)
\end{pmatrix} \label{eq:coext1},\\
\End_{\Triv(\Lambda)}(\Triv(\Lambda)\oplus M) & \cong \frac{\Triv(\Tria({}_{\Lambda}DM_E))}{\left( \begin{smallmatrix}
0 & 0 \\ 0 & DE\end{smallmatrix}\right)} \cong \begin{pmatrix}
\Triv(\Lambda) & DM \\ M & E
\end{pmatrix}. \label{eq:coext2}
\end{align}
\end{lemma}
\begin{proof}
For convenience, we denote simply by $\Tria:=\Tria({}_{\Lambda}DM_E)$ the triangular matrix ring.
The $\Tria$-bimodule $D\Tria$ can be written in matrix form as \[
\begin{pmatrix}
D\Lambda & 0 \\ M & DE
\end{pmatrix}.\]
Here, the left $\Tria$-action is defined by combining (i) the left $\Lambda$-action on $D\Lambda$, (ii) the left $E$-action on $DE$ and on $M$, and (iii) the map $DM\otimes_E M \to D\Lambda$ in \eqref{eq:dualbimod}; similarly, the right $\Tria$-action is defined by (i) the right $\Lambda$-action on $D\Lambda$ and on $M$, (ii) the right $E$-action on $DE$, and (iii) the map $M\otimes_{\Lambda} DM\to DE$ in \eqref{eq:dualbimod}.  It is routine to check that the vector-space isomorphism \eqref{eq:coext1} now lifts to an algebra isomorphism using the definition of trivial extension algebras.

For the gendo algebra over $\Triv(\Lambda)$, we can write out the endomorphism algebra in the form of the Morita context ring as explained in Example \ref{eg:Morita-context}
\[
\begin{pmatrix}
\End_{\Triv(\Lambda)}(\Triv(\Lambda)) & \Hom_{\Triv(\Lambda)}(M,\Triv(\Lambda))\\ \Hom_{\Triv(\Lambda)}(\Triv(\Lambda),M) & \End_{\Triv(\Lambda)}(M)
\end{pmatrix} \cong  \begin{pmatrix}
\Triv(\Lambda) & DM \\ M & E
\end{pmatrix}.
\]
Note that we used the natural isomorphism of the duality functors for a symmetric algebra to obtain the isomorphism in the upper right entry.
For the lower right entry, as $M$ is a graded $\Triv(\Lambda)$-module concentrated in degree $0$, endomorphisms of $M$ in $\mod\Triv(\Lambda)$ are the same as those in $\mod^\Z\Triv(\Lambda)$, hence, the same as those in $\mod \Lambda$.

Comparing the description of the defining bimodule homomorphisms of the Morita context ring corresponding to $\Triv(\Delta({}_{\Lambda}DM_E))$ in \eqref{eq:coext1} and that of the Auslander context in Example \ref{eg:Morita-context} (2), we see that the vector-space isomorphism is actually an algebra one.
\end{proof}

Before using Lemma \ref{lem:triv(tria)} to study gendo-symmetric algebras, we give a similar result which is interesting albeit not essential for the investigation.

\begin{proposition}\label{gendoisom}
Suppose $\Lambda$ is a finite-dimensional algebra and $M\in \mod\Lambda$ satisfying $\Hom_\Lambda(M,\Lambda)=0$.  Then the following holds.
\begin{enumerate}[\rm (1)]
\item There are algebra isomorphisms $\End_{\Lambda}(\Lambda\oplus M) \cong \begin{pmatrix}\Lambda & 0 \\ M & E \end{pmatrix}$ and $\Triv(\End_{\Lambda}(\Lambda\oplus M)) \cong \Triv(\Tria({}_\Lambda DM_E))$.

\item The global dimension of $\End_{\Lambda}(\Lambda\oplus M)$ is finite if and only if so is that of $\Tria({}_{\Lambda}DM_E)$.
\end{enumerate}
Moreover, its trivial extension is isomorphic to that of $\Tria({}_{\Lambda}DM_E)$.
\end{proposition}
\begin{proof}
(1) This follows by writing $\End_{\Lambda}(\Lambda\oplus M)$ in the form as in Example \ref{eg:Morita-context} (2) and use a similar argument as in the Proof of Lemma \ref{lem:triv(tria)}.

(2) For simplicity, take $\Gamma:=\End_{\Lambda}(\Lambda\oplus M)$ and $\Tria:=\Tria({}_{\Lambda}DM_E)$.  Since $\Triv(\Gamma)\cong \Triv(\Tria)$ by (1), it follows from Happel's equivalence \cite{H} that 
\[
\Db(\mod\Gamma)\simeq \stmod^\Z {\Triv(\Gamma)} \simeq \stmod^\Z{\Triv(\Tria)} \simeq \Db(\mod \Tria),
\]
where $\stmod^\Z\Triv(A)$ denotes the stable module category of the graded $\Triv(A)$-module when $\Triv(A)$ is graded by $\deg(A)=0$ and $\deg(DA)=1$.  Since finiteness of global dimension is preserved under derived equivalences, the claim follows.
\end{proof}
\begin{remark}
\begin{enumerate}
\item  For the reader who is a familiar with repetitive algebras, one can see using the matrix form of $\End_{\Lambda}(\Lambda\oplus M)$ in (1) that its repetitive algebra is the same as that of  $\Tria:=\Tria({}_{\Lambda}DM_E)$.

\item Combining with Lemma \ref{gendoisom} somewhat says that taking trivial extension and taking endomorphism algebra of generator ``commute up to a defect":
\[
\xymatrix@C=12pt{
&& \Lambda\oplus M \ar@{~>}[ld]_{\text{Triv}} \ar@{~>}[rd]^{\text{endo. alg.}}  & \\
&\Triv(\Lambda)\oplus M \ar@{~>}[d]_{\text{endo. alg.}} & & \End_{\Lambda}(\Lambda\oplus M)\ar@{~>}[d]^{\text{Triv}} \\
{\left(\begin{smallmatrix}0 & 0 \\ 0 &D\End_{\Lambda}(M)\end{smallmatrix}\right)} \ar@{^{(}->}[r]& \End_{\Triv(\Lambda)}\big(\Triv(\Lambda)\oplus M\big) \ar@{->>}[rr]  &&  \Triv(\End_{\Lambda}(\Lambda\oplus M))
}
\]
\end{enumerate}
\end{remark}

A special case of Lemma \ref{lem:triv(tria)} allows us to determine  the representation-finiteness of a gendo algebra with trivial extension base.

\begin{lemma}\label{lem:gendosemibrick}
Suppose $M = \bigoplus_{i=1}^r M_i$ is a $\Lambda$-module whose indecomposable direct summands $M_1, \ldots, M_r$ are pairwise Hom-orthogonal, i.e. $\Hom_{\Lambda}(M_i, M_j)\cong K$ for $i=j$; zero, otherwise.
Then the gendo-symmetric algebra $\End_{\Triv(\Lambda)}(\Triv(\Lambda)\oplus M)$ is representation-finite if and only if $\Tria({}_{\Lambda}DM_{\End_{\Lambda}(M)})$ is an iterated tilted algebra of Dynkin type.
\end{lemma}
\begin{proof}
By the assumption of $M$, the endomorphism algebra $E:=\End_{\Lambda}(M)$ is isomorphic to the semi-simple algebra $K^r$, which means that $\Triv(E)\cong (K[x]/(x^2))^r$.  In particular, it follows from Lemma \ref{lem:triv(tria)} that the gendo-symmetric algebra $\End_{\Triv(\Lambda)}(\Triv(\Lambda)\oplus M)$ is isomorphic to $\Triv(\Tria(DM))/I$, where $I$ is an ideal given by the direct sum of simple socles of projective(-injective) modules.  Hence, by repeatedly applying Drozd--Kirichenko rejection lemma (Lemma \ref{lem:reject}), $\Lambda$ is representation-finite if and only if so is $\Triv(\Tria(DM))$.  Now the claim follows from the fact that representation-finite trivial extension algebras are those of iterated tilted of Dynkin type \cite{AHR}.
\end{proof}

\begin{example}
Let $\Lambda:=K\overrightarrow{\mathbb{A}_n}$ be the path algebra of the linearly oriented Dynkin quiver of type $\mathbb{A}_n$, and $M=\bigoplus_{i=1}^n S_i$ be the direct sum of all simple $\Lambda$-modules.
Then $H:=\Tria(DM)\cong KQ/I$, where $Q$ is the quiver
\[
\xymatrix{
1\ar[r]^{a}\ar[d]^{b} & 2\ar[r]^{a}\ar[d]^{b} & \cdots \ar[r]^{a} & n-1 \ar[r]^{a}\ar[d]^{b} & n\ar[d]^{b} \\
1' & 2' & & (n-1)'  & n'
}
\]
and $I$ is generated by paths of the form $ab$.
This is an iterated tilted algebra of type $\mathbb{A}_{2n}$. Indeed,
denote by $P_x$ the indecomposable projective $H$-module corresponding to a vertex $x\in Q_0$, and define $T_{i'}:=\mathrm{Cok}(P_{i'}\xrightarrow{b\cdot-} P_i)$ for $i\in \{1, \ldots, n\}$, then $T:=\bigoplus_{i=1}^n P_i \oplus T_{i'}$ is a tilting $H$-module whose endomorphism algebra is $K\overrightarrow{\mathbb{A}_{2n}}$.
Hence, Lemma \ref{lem:gendosemibrick} tells us that the gendo-symmetric algebra $\End_{\Triv(\Lambda)}(\Triv(\Lambda)\oplus M)$ is representation-finite.  We remark that this gendo-symmetric algebra is a special gendo Brauer tree algebra studied in \cite{CM}.
\end{example}

\subsection{The case when $\Lambda$ is a path algebra}\label{subsec:A=kQ}

To understand representation-finite gendo-symmetric algebra with trivial extension base, we look at a special case of $\Tria(DM)$ in this subsection, namely, when $\Lambda=KQ$ is the path algebra of a finite acyclic quiver $Q$, and $M$ is an indecomposable $KQ$-module.  By Proposition \ref{prop:nuder}, up to almost $\nu$-stable derived equivalence of $\End_\Lambda(\Lambda \oplus M)$, it suffices to pick one $\Triv(KQ)$-module $M$ from each $\tau$-orbit.  Since $\Triv(KQ)$ is representation-finite, the pushdown functor $\Db(\mod KQ)\simeq \stmod^\Z \Triv(KQ) \to \stmod \Triv(KQ)$ is dense.  In particular, as $\{D(KQe_v)\mid v\in Q_0\}$ forms is a complete list of $\tau$-orbit representative of $\Db(\mod KQ)\simeq \stmod^\Z \Triv(KQ)$, it is also a complete list of $\tau$-orbit representative of $\stmod\Triv(KQ)$.  Thus, we only need to consider the case when $M$ is an indecomposable injective $KQ$-module, say, $M=D(\Lambda e_v)$ for $v\in Q_0$. 

Note that we have $\End_{\Lambda}(M)\cong K$ and $\Tria(DM)$ is the one-point coextension algebra of $\Lambda$ by $M$.  Hence, we have 
\begin{align}
\Tria(DM)\cong KQ'\;\;\text{ where }Q'_0 = Q_0 \sqcup \{ x \}\text{ and }Q'_1 = Q_1 \sqcup \{ v\to x\},\label{eq:extend}\\
\text{and }\;\;\End_{\Triv(KQ)}\big(\Triv(KQ)\oplus D(KQe_v)\big) \cong \Triv(KQ')/\soc P_x. \label{eq:Triv(KQ')/soc}
\end{align}
by Lemma \ref{lem:triv(tria)}.
Note that if $\Delta$ is the underlying graph of $Q$, then the underlying graph of $Q'$ is the extension graph that was denoted by $\Delta^v$ in Section \ref{sec:intro}.

For ease of reference, we label the vertices of the Dynkin graphs as follows.
\begin{align}
\mathbb{A}_n \ (n\geq 1): \quad& \xymatrix{ 1\ar@{-}[r] & 2\ar@{-}[r] & 3 \ar@{-}[r]& \cdots \ar@{-}[r]& n-2\ar@{-}[r]& n-1\ar@{-}[r] & n} \notag\\
\mathbb{D}_n \ (n\geq 4): \quad& \xymatrix{ 1\ar@{-}[r] & 2\ar@{-}[r] & 3 \ar@{-}[r]& \cdots \ar@{-}[r]& n-2\ar@{-}[r]\ar@{-}[d] & n-1\\ &&&&n& } \label{eq:Dn}\\
\mathbb{E}_n \ (n=6,7,8): \quad& \xymatrix{ 1\ar@{-}[r] & 2\ar@{-}[r] & 3 \ar@{-}[r]\ar@{-}[d] &\cdots \ar@{-}[r]& n-2\ar@{-}[r] & n-1\\ & & n & & } \notag 
\end{align}

We give a complete list of $Q$ and $M$ with $\End_{\Triv(KQ)}(\Triv(KQ)\oplus M)$  representation-finite.

\begin{lemma}\label{finitetype}
Let $Q$ be an acyclic quiver and $v$ its vertex.
Then the following are equivalent.
\begin{enumerate}
\item $\End_{\Triv(KQ)}(\Triv(KQ)\oplus D(KQe_v))$ is representation-finite.
\item The quiver $Q'$ of \eqref{eq:extend} is simply-laced Dynkin.
\item The pair $(Q, v)$ satisfies one of the following.
\begin{enumerate}
\item $Q$ is of type $\mathbb{A}_n$ with $v\in \{1, 2, n-1, n\}$ or $(n,v)\in \{(5,3), (6,3), (6,4), (7,3), (7,5)\}$.
\item $Q$ is of type $\mathbb{D}_n$ with $v=1$ or $(n,v)\in\{(n,n-1), (n,n)\mid n=4, 5, 6, 7\}$.
\item $Q$ is of type $\mathbb{E}_n$ with $(n,v)\in\{(6,1), (6,5), (7,5)\}$.
\
\end{enumerate}
\end{enumerate}
\end{lemma}
\begin{proof}
\underline{(1)$\Leftrightarrow$(2)}: Since $D(KQe_v)$, we always have simple endomorphism ring, and so $\Tria(KQe_v)\cong KQ'$ as explained above.  It follows from Lemma \ref{lem:gendosemibrick} that (1) holds if and only if $Q'$ is a simply laced Dynkin quiver.

\underline{(2)$\Leftrightarrow$(3)}: This is a simple combinatorial exercise from the construction \eqref{eq:extend} of $Q'$.
\end{proof}

Consequently, we have a characterisation for representation-finiteness of gendo-symmetric algebras with trivial extension base.

\begin{proposition}\label{classify-triv-ext}
Let $B$ be a representation-finite symmetric algebra of type $(\Delta,1)$, and $M$ be an indecomposable non-projective $B$-module.  Then the following are equivalent. 
\begin{enumerate}
\item $\Gamma:=\End_B(B\oplus M)$ is representation-finite.
\item The extension $\Delta^v$ of $\Delta$ at the vertex $v:=v(M)\in \Delta_0$ is Dynkin.
\item $\Gamma$ is almost $\nu$-stable derived equivalent to $\Triv(KQ')/\soc(P_x)$, where $Q'$ is a Dynkin quiver whose underlying graph is an extension $\Delta^u$ of $\Delta$ at some vertex $u\in \Delta_0$.
\end{enumerate}
\end{proposition}
\begin{proof}
As mentioned, the indecomposable injective $KQ$-modules form a complete set of $\tau$-orbit representatives of indecomposable non-projective $\Triv(KQ)$-modules.  Hence, we have a stable equivalence $F:\stmod B\to \stmod\Triv(KQ)$ with $F(M)=D(KQe_v)$.  By Proposition \ref{prop:nuder} (3), it follows that $\Gamma$ is representation-finite if and only if so is $\Gamma'$.  Thus, the equivalence between (1) and (2) follows immediately from the first two equivalent conditions in Lemma \ref{finitetype}.

Suppose (2) holds.  Using the stable equivalence $F$ in the previous paragraph and the isomorphism $\Gamma'\cong \Triv(KQ')/\soc(P_x)$ from Lemma \ref{lem:triv(tria)}, we get (3) as required.

Suppose (3) holds.  Since $Q'$ is Dynkin, $\Triv(KQ')$ is representation-finite symmetric of type $(Q',1)$.  It follows by rejection lemma \ref{lem:reject} that $\Triv(KQ')/\soc(P_x)$ is representation-finite.  Representation-finiteness of $\Gamma$ then follows from Proposition \ref{prop:main(3)=>(1)}.
\end{proof}

We obtain a similar list as Lemma \ref{finitetype} for determining infinite-tame type.

\begin{proposition}\label{tametype}
Let $Q$ be an acyclic quiver and $v$ its vertex.
Then the following are equivalent.
\begin{enumerate}
\item $\End_{\Triv(KQ)}(\Triv(KQ)\oplus D(KQe_v))$ is of infinite-tame type.
\item The quiver $Q'$ of \eqref{eq:extend} is simply-laced Euclidean non-Dynkin.
\item The pair $(Q,v)$ satisfies one of the following.
\begin{enumerate}
\item $Q$ is of type $\mathbb{A}_n$ with $(n,v)\in\{(7,3), (7,4), (8,3), (8,5)\}$.
\item $Q$ is of type $\mathbb{D}_n$ with $v=2$ or $(n,v)\in\{(8,7), (8,8)\}$.
\item $Q$ is of type $\mathbb{E}_n$ with $(n,v)\in \{(6,6), (7,1), (8,7)\}$.
\end{enumerate}
\end{enumerate}
\end{proposition}
\begin{proof}
For the equivalence between (1) and (2), the argument is almost the same as in the proof of Lemma \ref{lem:gendosemibrick}. The modification needed is to replace the use of Lemma \ref{lem:gendosemibrick} by Theorem \ref{tricot} (combining with Lemma \ref{lem:triv(tria)} and Rejection lemma).
\end{proof}

\section{Modified Brauer tree algebra as base algebra}\label{sec:modBTree}

In this section, we investigate gendo-symmetric algebras of the form $\End_B(B\oplus M)$ where $B$ is a modified Brauer star algebra.  As a historical note, this base algebra was not given any specific name when it was first investigated in \cite{W}; we took the name `modified Brauer tree algebra' that was used in \cite{Sko}.  The modified Brauer star algebra is the modified Brauer tree algebra associated to a (Brauer) star.  Let us recall its definition now.

\begin{definition}\label{def:modBSt}
Let  $n\geq2$.  The \emph{modified Brauer star algebra} with $n$ simples and parameter $\epsilon\in \{0,1\}$, denoted by $B_n^{\epsilon}$, is presented by the following quiver with relations:

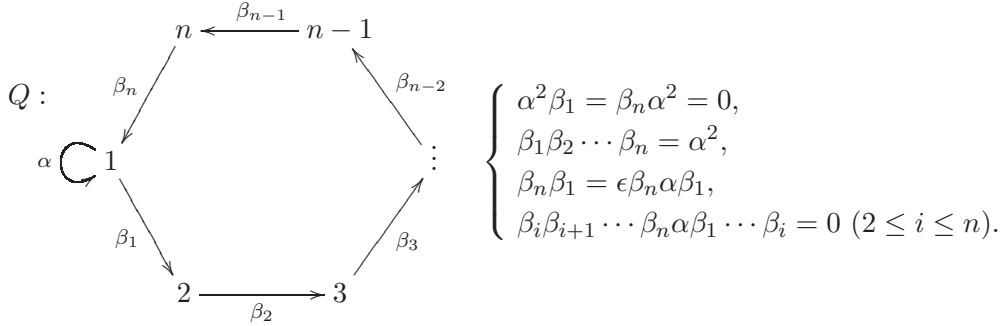
\[\begin{array}{cc}
\vcenter{\xymatrix@C=15pt@!R=1pt{
&&n \ar[ldd]_{\beta_n} &&\ar[ll]_{\beta_{n-1}}n-1&\\
Q: &&&&&\\
& \ar@(ul,dl)_{\alpha} 1 \ar[rdd]_{\beta_1} &&&&\ar[luu]_{\beta_{n-2}} \vdots  \\\\
&&2 \ar[rr]_{\beta_2} &&3 \ar[ruu]_{\beta_3}&}}
&
\begin{cases}
\ \alpha^2\beta_1=\beta_n\alpha^2=0, \\
\ \beta_1\beta_2\cdots \beta_n=\alpha^2, \\
\ \beta_n\beta_1=\epsilon \beta_n\alpha\beta_1, \\
\ \beta_i \beta_{i+1} \cdots \beta_n \alpha \beta_1 \cdots \beta_i = 0 ~~ ( 2 \leq i \leq n ).
\end{cases}
\end{array}\]
\end{definition}
As mentioned in Section \ref{sec:prelim}, $B_n^0$ and $B_n^1$ are isomorphic if and only if the characteristic of the ground field $K$ is not $2$. 

The following is the aim of this section.

\begin{proposition}\label{GSofNRBT}
Let $M$ be an indecomposable non-projective $B_n^{\epsilon}$-module.
Then the following are equivalent:
\begin{enumerate}
\item The endomorphism algebra $\Gamma:= \End_{B_n^{\epsilon}}(B_n^{\epsilon} \oplus M)$ is representation-finite. 

\item $M$ belongs to the $\tau$-orbit of $\rad P_i$ for some $i\neq1$.
\end{enumerate}
In such a case, $\Gamma$ is almost $\nu$-stable derived equivalent to $B_{n+1}^{\epsilon}/\soc P_i$ for any $i \neq 1$.
\end{proposition}

Let us start by fixing the $\tau$-orbit representatives of $B_n^\epsilon$.  Recall first that the stable AR-quiver of $B_n^\epsilon$ is given by the translation quiver $\mathbb{Z}\mathbb{D}_{3n}/\langle\tau^{2n-1}\rangle$. 
Using the labelling of the vertices of the Dynkin diagram $\mathbb{D}_{3n}$ in subsection \ref{subsec:A=kQ}, it will be convenience to call the $\tau$-orbit of non-projective modules the \emph{$i$-th row} (of the stable AR-quiver), for $i$ the corresponding vertex in $\mathbb{D}_{3n}$.  We also say that $M$ is \emph{in row $i$} in such a case.  We will use the following $\tau$-orbit representatives.

\begin{proposition}{\rm  \cite[Section 4]{W}}\label{modMBTAlg}
The following indecomposable modules form a complete list of $\tau$-orbit representatives of indecomposable non-projective $B_n^{\epsilon}$-modules.
\begin{itemize}
\item \underline{Row 1}: The radical $\rad P_i$ (which is denoted by $E_{i,n}$ in \cite{W}) of the indecomposable projective module $P_i$ corresponding to any vertex $i\in\{2,3, \ldots, n\}$. 

\item \underline{Row $n-r+1$ for $r\in\{n, n-1, \ldots, 2\}$}: The module $H_{r,1}$ which can be described by the $Q$-colored quiver
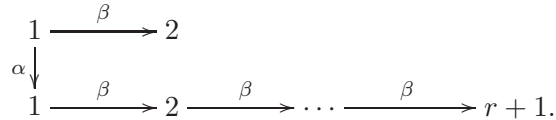
\[
\xymatrix@R=15pt@C=40pt{
1 \ar[r]^{\beta} \ar[d]_{\alpha}& 2&& \\
1 \ar[r]^{\beta} & 2\ar[r]^{\beta} & \cdots \ar[r]^{\beta}& **[r]{r+1.}
}
\]

\item \underline{Row $n+r+s$ with $0\leq r < n$, $0\leq s < n$}: The module $L_{r,s}$, which can be described by the $Q$-colored quiver
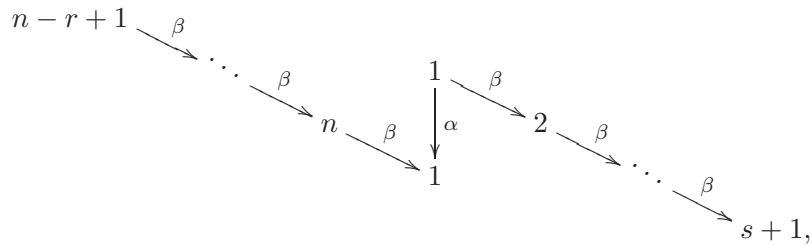
\[
\xymatrix@!0@R=20pt@C=40pt{**[l]{n-r+1}\ar[rd]^{\beta}&&&&&\\
&\ddots \ar[rd]^{\beta}&&1\ar[dd]^{\alpha} \ar[rd]^{\beta}  &&\\
&&n\ar[rd]^{\beta}&&2 \ar[rd]^{\beta}&  \\
&&&1&&\ddots \ar[rd]^{\beta} \\
&&&&&& **[r]{s+1},}
\]
where $n-r+1$ is regarded as $1$ when $r=0$. Note that multiple distinct $L_{r,s}$'s can belong to the same $\tau$-orbit.  We can take, for example, $L_{0,1}, L_{0,2},\ldots, L_{0,n-1}, L_{1,n-1}, \ldots, L_{n-1,n-1}$ as the list of (representatives of) pairwise distinct $\tau$-orbits.

\item \underline{Row $3n-1$}: The simple module $S_1$ (which is denoted by $E_{1,0}$ in \cite{W}) corresponding to the vertex $1$ of $Q$.

\item \underline{Row $3n$}: The module $D:= P_1 / \alpha\beta_1 P_1 \cong \alpha P_1$ (which is denoted by $P$ in \cite{W}).  It can be also described by the $Q$-colored quiver
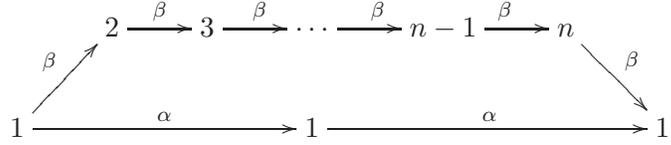
\[
\xymatrix{
& 2 \ar[r]^{\beta} & 3 \ar[r]^{\beta} & \cdots \ar[r]^{\beta} &n-1 \ar[r]^{\beta}& n \ar[rd]^{\beta} &\\
1 \ar[ru]^{\beta}\ar[rrr]^{\alpha}&&&1\ar[rrr]^{\alpha}&&& 1.
}
\]
\end{itemize}
\end{proposition}

We stress that the $Q$-colored quiver description above is independent of $\epsilon\in\{0,1\}$.  On the other hand, the structure of $\rad P_i$ depends on $\epsilon$, but it will not affect our arguments to come.  For the ease of reader, we will display the explicit AR-quiver of $B_2^\epsilon$ after the following lemma.

\begin{lemma}\label{EndMBTArad}
For each $1<i\leq n$, we have an algebra isomorphism
\[
\End_{B_n^{\epsilon}}(B_n^{\epsilon}\oplus\rad P_i)\cong B_{n+1}^{\epsilon}/\soc P_{i+1}.
\]
In particular, the gendo-symmetric algebra $\End_{B_n^{\epsilon}}(B_n^{\epsilon}\oplus\rad P_i)$ is  representation-finite.
\end{lemma}

\begin{proof}
First observe that $\rad P_i$ has simple endomorphism ring $\End_{B_n^{\epsilon}}(\rad P_i)\cong  K$, so the vertex corresponding to $\rad P_i$ in the quiver of $\Gamma:=\End_{B_n^{\epsilon}}(B_n^{\epsilon}\oplus\rad P_i)$ has no loop attached to it.  Moreover, since any homomorphism $P_j\to P_i$ must factor through $\rad P_i$, the quiver of $\Gamma$ coincides with the quiver $Q$ of $B_{n+1}^\epsilon$ - namely, the vertices $j\leq i$ in $Q$ corresponds to $P_j$, the vertex $i\in Q_0$ corresponds to $\rad P_i$, and all remaining vertices $i<j\leq n+1$ corresponds to $P_{j-1}$.

Now that we have an surjective algebra homomorphism $KQ\to \Gamma$, it remains to verify the Cartan matrices of the two algebras coincide.  This is immediate for the entries involving the projective $B_n^\epsilon$-modules.  The others are also easy to check - for example, $\dim \Hom_{B_n^\epsilon}(P_{j}, \rad P_i)=\dim \Hom_{B_{n+1}^\epsilon}(P_{j},  P_i)$ for all $j=j'<i$.  We leave the rest as an exercise.

Finally, with the isomorphism verified, representation-finiteness of $\Gamma$ follows from Rejection lemma \ref{lem:reject}.
\end{proof}

We demonstrate the proof of Proposition \ref{GSofNRBT} in the case when $n=2$ first to let the reader be familiar with the various indecomposable modules appearing in Proposition \ref{modMBTAlg}.  Then we will show how the case of $n>2$ be can reduced to the $n=2$ case.

Until further notice, we set $B:=B_2^\epsilon$.

We start by drawing the AR-quiver of $B_2^\epsilon$ in the following;
note that indecomposable $\Lambda$-modules can be uniquely identified by its Loewy structure, so the vertices of the AR-quiver are labelled by the Loewy diagrams.

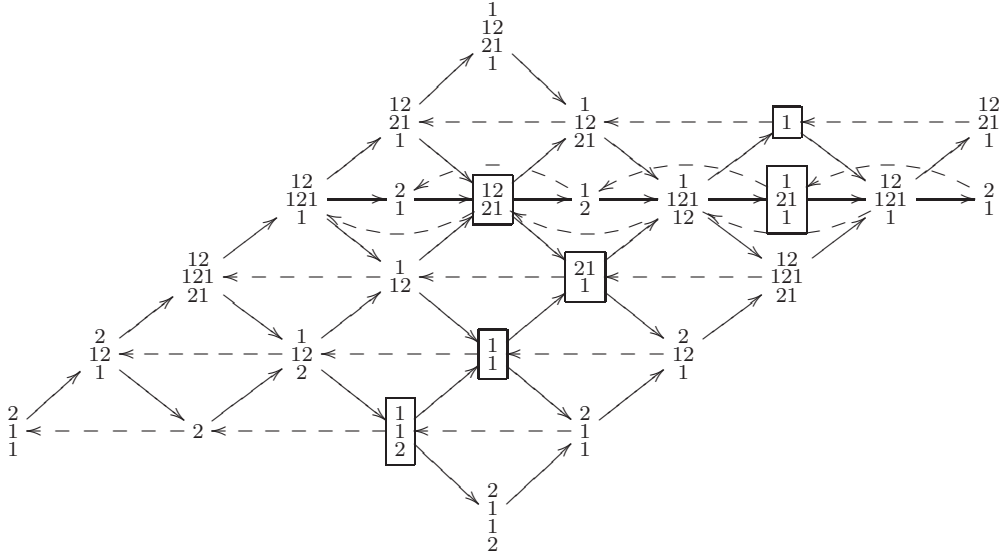
\[
\xymatrix@C=20pt@R=4pt{
&& &   &         &\modUle{1}{12}{21}{1} \ar[rd]  & &  & & \\
&&&   &\moDule{12}{21}{1}  \ar[ru] \ar[rd]& &\moDule{1}{12}{21} \ar@{-->}[ll] \ar[rd]&  & *+[F]{\Module{1}} \ar[rd] \ar@{-->}[ll]&& \moDule{12}{21}{1} \ar@{-->}[ll] \\
&&  &\moDule{12}{121}{1}  \ar[r] \ar[ru] \ar[rd]&  \mOdule{2}{1} \ar[r] &*+[F]{\mOdule{12}{21}} \ar@/^13pt/@{-->}[ll] \ar[r] \ar[ru] \ar[rd]& \mOdule{1}{2} \ar[r] \ar@/_13pt/@{-->}[ll] & \moDule{1}{121}{12} \ar@/^13pt/@{-->}[ll]\ar[rd]\ar[ru] \ar[r]& *+[F]{\moDule{1}{21}{1}}  \ar[r] \ar@/_13pt/@{-->}[ll] & \moDule{12}{121}{1} \ar[r] \ar[ru] \ar@/^13pt/@{-->}[ll] & \mOdule{2}{1} \ar@/_13pt/@{-->}[ll] \\
&&\moDule{12}{121}{21} \ar[ru] \ar[rd]&   & \ar@{-->}[ll]\mOdule{1}{12} \ar[ru] \ar[rd]&   &*+[F]{\mOdule{21}{1}} \ar[ru] \ar@{-->}[ll] \ar[rd]&  & \ar@{-->}[ll] \moDule{12}{121}{21} \ar[ru] & & \\
& \moDule{2}{12}{1} \ar[ru] \ar[rd]  &  & \ar@{-->}[ll] \moDule{1}{12}{2} \ar[ru] \ar[rd]&   &  *+[F]{\mOdule{1}{1}} \ar@{-->}[ll] \ar[ru] \ar[rd]&   &\moDule{2}{12}{1} \ar[ru] \ar@{-->}[ll]&  & & \\
  \moDule{2}{1}{1} \ar[ru] & & \ar@{-->}[ll] \Module{2} \ar[ru] &   &*+[F]{\moDule{1}{1}{2}} \ar@{-->}[ll] \ar[ru] \ar[rd]&   & \ar@{-->}[ll] \moDule{2}{1}{1} \ar[ru] &  && & \\
&&  &&   & \modUle{2}{1}{1}{2} \ar[ru] &&& & 
}
\]

We will carry out a case-by-case calculation, one indecomposable module $M$ for each row ($\tau$-orbit) to determine the representation-finiteness of $A:=\End_B(B\oplus M)$.
The indecomposable non-projective modules we pick for $M$ are the ones framed in the AR-quiver above.

Clearly, the quiver of $A$ has one extra vertex compare to that of $B$.  We will label this new vertex by $M$.  Note that the corresponding primitive idempotent $e_M\in A=\End_B(B\oplus M)$ is given by the composition of the natural projection and inclusion:
\[
e_M: B\oplus M \twoheadrightarrow M \hookrightarrow B\oplus M.
\]
For simplicity, we also let $e:=e_1+e_M\in A$.

\begin{caseM}[$M=\rad P_2$]\label{case-RF}
By Lemma \ref{EndMBTArad}, $A$ is representation-finite.
\end{caseM}
\begin{caseM}[$M=L_{0,0}=\mOdule{1}{1}$]\label{case-iv}
Note that $L_{0,0}$ is in the second row the stable AR-quiver. The projection $P_1\to M$ and inclusion $M\to P_1$ yields two irreducible maps in $\End_\Lambda(P_1\oplus M)\cong e\Gamma e$.  The left-multiplication map $\alpha\cdot-:P_1\to P_1$ induces an endomorphism $\gamma:M\to M$ which is irreducible in $e\Gamma e$.  Hence, the $\Ext^1$-quiver of $e\Gamma e$ contains a subquiver $Q$ of the form
\[
Q:\quad  \vcenter{\xymatrix@C=70pt
{1 \ar@(ul,dl)_{\alpha} \ar@/_10pt/[r]_{\beta_1} &M \ar@(ur,dr)^{\gamma} \ar@/_10pt/[l]_{\beta_2}  
}}.
\]
In fact, this is the whole $\Ext^1$-quiver of $e\Gamma e$.  Anyway, since the separated quiver of $Q$ is of type $\widetilde{\mathbb{A}_3}$, it follows from Proposition \ref{prop:RFRad} that $e\Gamma e$ is representation-infinite, and hence so is $\Gamma$.
\end{caseM}

\begin{caseM}[$M=L_{1,0}=\mOdule{21}{1}$]\label{case-iv-2}
$e\Gamma e\cong \End_\Lambda(P_1\oplus M)$ has the same ordinary quiver as in the previous case \eqref{case-iv} and so the same argument applies.  Hence, $\Gamma$ is representation-infinite.
\end{caseM}
\begin{caseM}[$M=L_{1,1}=\mOdule{12}{21}$]\label{case-iii}
We see that the local algebra $\End_\Lambda(L_{1,1}) \cong e_M \Gamma e_M$ has two generators; equivalently, $e_M\Gamma e_M$ is presented by the ordinary quiver $\xymatrix@1{ \bullet \ar@(ru,rd)^{} \ar@(lu,ld)_{} }$.
Namely, one generator is induced by the action of $\alpha$, and the other corresponds to mapping the composition factor $S_2$ in the top of $M$ to that in the socle.
By Example \ref{Exrad2}, $e_M\Gamma e_M$ is representation-infinite, which means that so is $\Gamma$.
\end{caseM}
\begin{caseM}[$M=S_1$]
\label{case-i}
Similar to the case \eqref{case-iv}, the projection $P_1\to S_1$ and inclusion $S_1\to P_1$ contribute two irreducible maps in $\Gamma$.  It is easy to see that neither of these factor through $P_2$.  So the $\Ext^1$-quiver of $\Gamma$ contains a subquiver $Q$ of the form
\[
Q_\Gamma:\vcenter{
\xymatrix@C=30pt@R=15pt
{x \ar@/_10pt/[r]_{\gamma_1} & \ar@/_10pt/[l]_{\gamma_2} 1 \ar@(ur,ul)_{\alpha} \ar@/_10pt/[r]_{\beta_1}  & \ar@/_10pt/[l]_{\beta_2} 2
}}.
\]
In fact, this is the whole of the $\Ext^1$-quiver of $\Gamma$.  In any case, the separated quiver of $\Gamma$ is of type $\widetilde{\mathbb{D}}_5$, and so $\Gamma$ is representation-infinite by Proposition $\ref{prop:RFRad}$.
\end{caseM}
\begin{caseM}[$M=D=\moDule{1}{12}{1}$]\label{case-ii}
The order 2 automorphism on $\mathbb{D}_{6}$ induces an autoequivalence on the stable module category $\stmod\Lambda$; see \cite{Asa2}.  Composing this with a sufficient power of $\tau$ yields an autoequivalence $F$ of $\stmod\Lambda$ that satisfies $F(D)=S_1$.  Hence, it follows from Proposition \ref{prop:nuder} and the previous case \eqref{case-i} that $\Gamma$ is representation-infinite.
\end{caseM}

Let us now go back to the general case $B_n^\epsilon$ and prove Theorem \ref{GSofNRBT}.  We recall the following useful lemma to help us reduce to easier situations.

\begin{lemma}\label{IdempRingIsom}
Let $M$ be a module over an algebra $\Lambda$ and $e$ an idempotent of $\Lambda$. 
If $M$ is isomorphic to $Me\Lambda$,  
then
$(e+e_M)\End_\Lambda(\Lambda\oplus M)(e+e_M)$
is isomorphic to 
$\End_{e\Lambda e}(e\Lambda e \oplus Me)$
as an algebra. 
\end{lemma}

\begin{proof}
Since $X\cong Xe\Lambda = Xe\otimes_{e\Lambda e}e\Lambda$, we have bifunctorial isomorphisms
\[
\Hom_\Lambda(X,Y) = \Hom_\Lambda(Xe\otimes_{e\Lambda e}e\Lambda, Y)\cong \Hom_{e\Lambda e}(Xe, \Hom_\Lambda(e\Lambda,Y)) = \Hom_{e\Lambda e}(Xe, Ye).
\]
By the assumption, we can take $X=Y=M\oplus e\Lambda$ which yields $(e+e_M)\End_\Lambda(X)(e+e_M)\cong \End_{\Lambda}(M\oplus e\Lambda) \cong \End_{e\Lambda e}(e\Lambda e\oplus Me)$; note that this is indeed an algebra (iso)morphism thanks to functoriality.
\end{proof}

We now proceed to prove the main result of the section.

\begin{proof}[Proof of Proposition \ref{GSofNRBT}]

Let $M$ be one of the indecomposable $B_n^\epsilon$-modules listed in Proposition \ref{modMBTAlg}.
For simplicity, denote by $B$ and $A$ the algebras $B_n^\epsilon$ and $\End_{B_n^{\epsilon}}(B_n^{\epsilon}\oplus M)$ respectively.
We keep the notation $e_M$ and $e_i \in B$ with $i\in\{1,2,\ldots,n\}$ for the primitive idempotents of $A$ as in the case of $n=2$.

\underline{(2)$\Rightarrow$(1)}: By Proposition \ref{prop:nusta} (3), we can take $M=\rad P_i$ for any $i\neq 1$.  Then Lemma \ref{lem:reject} and Lemma \ref{EndMBTArad} says precisely that the resulting gendo-symmetric algebra $A$ is representation-finite.

Note that the almost $\nu$-stable derived equivalence statement follows from the description of $A=\End_B(B\oplus M)$ in Lemma \ref{EndMBTArad}.

\underline{(1)$\Rightarrow$(2)}: We will show that when $M$ is one of the representatives in Proposition \ref{modMBTAlg} and not isomorphic to $\rad P_i$ for any $i\neq 1$, then the resulting gendo-symmetric algebra $A$ is representation-infinite.  Then the implication follows by Proposition \ref{prop:nuder}.

To show $A$ is representation-infinite, we will pick an idempotent $e\in A$ so that Lemma \ref{IdempRingIsom} can be applied.  This yields an algebra isomorphism $(e+e_M)A(e+e_M) \cong \End_{B_{2}^\epsilon}(B_{2}^\epsilon \oplus Me)$ with $Me$ indecomposable.  The idempotent $e$ and the indecomposables $M, Me$ are the following:
\[
\begin{array}{c|cccccc}
M & H_{r,1} & L_{0,0} & L_{t,0} & L_{n-1,s} & S_1 & D \\
 & \scriptstyle 2\leq r\leq n &  & \scriptstyle 1\leq t<n & \scriptstyle 1\leq s<n & & 
\\ \hline
e & e_1+e_2 & e_1+e_2 & e_1+e_{n-t+1} & e_1+e_2 & e_1+e_2 & e_1+e_2 \\ \hline
Me & L_{0,0} & L_{0,0} & L_{1,0} & L_{1,1} & S_1 & D.
\end{array}
\]
Now our claim follows from the calculation done in the case of $B_2^\epsilon$.
\end{proof}

\begin{corollary}\label{classify-modiBr}
Let $B$ be a modified Brauer tree algebra with $|B|=n$, $M$ be an indecomposable non-projective $B$-module, and $\Gamma:=\End_B(B\oplus M)$ be the associated gendo-symmetric algebra. Then the following are equivalent: 
\begin{enumerate}
\item $\Gamma$ is representation-finite.
\item $M$ is in the $\tau$-orbit given by the vertex $1$ of $\mathbb{D}_{3n}$.
\item $\Gamma$ is almost $\nu$-stable derived equivalent to $B_{n+1}^\epsilon/\soc P_i$ for some $i\neq1$.
\end{enumerate}
\end{corollary}
\begin{proof}
By Proposition \ref{prop:nuder}, there is a stable equivalence $F:\stmod B\to \stmod B_n^\epsilon$ so that representation-finiteness of $\Gamma$ is equivalent to that of $\End_{B_n^\epsilon}(B\oplus F(M))$. 
As explained in subsection \ref{subsec:RFS}, $F$ induces an automorphism on the $\tau$-orbits of indecomposable non-projective $B$-modules.  Since there is only one non-trivial graph automorphism $\phi$ on $\mathbb{D}_{3n}$, $F$ preserves all $\tau$-orbits with the possible exception of the two labelled by vertices $3n-1$ and $3n$.  On the other hand, as the $\tau$-orbit of $\rad P_i$ for any $i\neq 1$ is the one labelled by the vertex $1$ of $\mathbb{D}_{3n}$, the equivalence between (1) and (2) now follows from Proposition \ref{GSofNRBT}.

Suppose (2) holds.  Then using the stable equivalence $F$ in the previous paragraph and Lemma \ref{EndMBTArad}, we get (3) as required. 

Suppose (3) holds.  As $B_{n+1}^\epsilon$ is representation-finite, so is its quotient $B_{n+1}^\epsilon/\soc P_i$.  Hence, by Proposition \ref{prop:main(3)=>(1)}, $\Gamma$ is also representation-finite.
\end{proof}

\section{Main result and remarks}\label{sec:final}
\subsection{Proof of the main result}\label{subsec:proof}
Let us first show the analogue of Proposition \ref{classify-triv-ext} and Corollary \ref{classify-modiBr} for the case of Brauer tree algebras.

In the following, we denote by $N_{n,m}$ the symmetric Nakayama algebra with $n$ simples and Loewy length $mn+1$.  For detailed background of these algebras, and also of Brauer tree algebras, we refer to \cite{CM}.

The most important step we need is already done in \cite{B}, which we recall below; see also the next subsection \ref{subsec:conclu} for comparison of the strategies of its proof and ours in the case of RFSy-type $(\Delta,1)$ and $(\mathbb{D}_{3r},1/3)$.

\begin{proposition}{\rm \cite[Ch. 10, 11]{B}}\label{symmNaka case}
Let $B$ be a symmetric Nakayama algebra with $n\geq 1$ simples and Loewy length $nm+1$ for some $m>1$.
Let $M$ be an indecomposable non-projective $B$-module.
Then $\End_B(B\oplus M)$ is representation-finite if and only if $M$ is a simple module or the radical of an indecomposable projective module.
\end{proposition}

It will be convenient to reformulate the above result as follows. 

\begin{corollary}\label{classify-Br}
Let $B$ be a representation-finite symmetric algebra of type $(\mathbb{A}_{mn},1/m)$ for some $m>1$ and $n\geq 1$, i.e. $B$ is a Brauer tree algebra with $n\geq 1$ simples and multiplicity $m>1$.
Let $\Gamma:=\End_B(B\oplus M)$ be the gendo-symmetric algebra associated to an indecomposable non-projective $B$-module $M$.
Then the following are equivalent: 
\begin{enumerate}
\item $\Gamma$ is representation-finite.
\item $v:=v(M)\in \{1,mn\}$, i.e. $M$ is a module on the mouth the the stable AR-quiver.
\item $\Gamma$ is almost $\nu$-stable derived equivalent to $N_{n+1,m}/\soc P$ for any indecomposable projective $N_{n+1,m}$-module $P$.
\end{enumerate}
\end{corollary}
\begin{proof}
By Proposition \ref{prop:nuder} and Theorem \ref{thm:asa}, there is a stable equivalence $F:\stmod B\to \stmod N_{n,m}$ such that $\Gamma$ is almost $\nu$-stable derived equivalent to $\Gamma':=\End_{N_{n,m}}(N_{n,m}\oplus F(M))$, and $\Gamma$ is representation-finite if and only if so is $\Gamma'$.

It follows from Proposition \ref{symmNaka case} that $\Gamma'$ is representation-finite if and only if $F(M)$ is a simple or a radical of an indecomposable projective $N_{n,m}$-module, which is then equivalent to $v(F(M))\in\{1,nm\}$ (c.f. \cite[Sec 2.3]{CM}).  Since $F$ either fixes all $\tau$-orbits, or permutes them in the way swaps the orbit labelled by $i$ with that labelled by $n-i$, we have that $v(F(M))\in \{1,mn\}$ if and only if so is $v(M)$.  This proves the equivalence between (1) and (2).

Suppose (2) holds.  Then $F(M)$ is either $\rad(Q)$ for some indecomposable projective $N_{n,m}$-module $P$, or a simple $N_{n,m}$-module $S$.  In the latter case, we replace $F$ by $\Omega^{-1}\circ F$ where $\Omega^{-1}$ is the cosyzygy functor.  Hence, we have $\Gamma' = \End_{N_{n,m}}(N_{n,m}\oplus \rad(Q))$ for some indecomposable projective $Q$.  An easy calculation shows that $\Gamma'\cong N_{n+1,m}/\soc P$ for any indecomposable projective module $P$; see \cite[Theorem 3.9]{CM} for the details.  This shows the implication from (2) to (3).

Since representation-finiteness of $N_{n+1,m}$ implies that of $N_{n+1,m}/\soc P$, it follows from Proposition \ref{prop:main(3)=>(1)} that (1) holds.
\end{proof}

The main result of each section now combines to the required proof of the main Theorem \ref{thm:main}.

\begin{proof}[Proof of Theorem \ref{thm:main}]
For simplicity, denote by $\Gamma$ the endomorphism ring $\End_B(B\oplus M)$.

\underline{(1)$\Rightarrow$(2)}:  Since $B\cong e\Gamma e$ for the idempotent given by the natural map $B\oplus M \twoheadrightarrow B\hookrightarrow B\oplus M$, $\Gamma$ being representation-finite means that so is $B$.  Now the condition of $v$ for each case of $(\Delta,f)$ follows from the implications `(1)$\Rightarrow$(2)' in Proposition \ref{classify-triv-ext} for trivial extension type $(\Delta,1)$, in Corollary \ref{classify-modiBr} for modified Brauer tree type $(\mathbb{D}_{3n},1/3)$, and in Corollary \ref{classify-Br} for type $(\mathbb{A}_{mn},1/m)$.

\underline{(2)$\Rightarrow$(3)$\Rightarrow$(1)}:  This is precisely the implications `(2)$\Rightarrow$(3)$\Rightarrow$(1)' in Proposition \ref{classify-triv-ext} for trivial extension type $(\Delta,1)$, in Corollary \ref{classify-modiBr} for modified Brauer tree type $(\mathbb{D}_{3n},1/3)$, and in Corollary \ref{classify-Br} for type $(\mathbb{A}_{mn},1/m)$.
\end{proof}

\subsection{Concluding remarks}\label{subsec:conclu}

\begin{enumerate}[leftmargin=5.5mm,itemsep=0.5cm]
\item Unifying strategy of proof and generalisation: 
The strategy in \cite{B} for determining the representation-infiniteness of $\Gamma:=\End_B(B\oplus M)$ for $B$ a symmetric Nakayama algebra and $v(M)\notin\{1,mn\}$ is given in two steps.  In the first step, one calculates the universal covering $\widetilde{\Gamma}$ of $\Gamma$ explicitly.  Then one tries to find a subquiver $Q$ of the quiver of $\widetilde{\Gamma}$ (or of $\Gamma$) such that $KQ$ is representation-infinite.  In this specific case of $B$ and $M$, it turns out that one can always find a subquiver of type $\tilde{\mathbb{D}}_n$.  

There is some similarities of the strategy of \cite{B} with ours in the case of trivial extensions - we replace their first step by showing an alternative description of $\Gamma$ that is somewhat like an orbit algebra, and replace their second step by looking at one-point coextension algebra (the algebra $\Tria(DM)$).  On the other hand, the strategy we used in this article for the case of RFSy-type $(\mathbb{D}_{3n},1/3)$ is rather ad-hoc.  It is desirable to unifying the different strategies for the three types of representation-finite symmetric algebras.

In view of this problem, we speculate that there is a generalised version of Lemma \ref{lem:triv(tria)}, where trivial extension is replaced by \emph{orbit algebra $\hat{\Lambda}/\langle\varphi\rangle$ of a repetitive algebra $\hat{\Lambda}$} for some `nice enough' automorphism $\phi$ of $\hat{\Lambda}$. 
Note that the case of trivial extension has $\Lambda=KQ$ and $\varphi$ being the Nakayama automorphism on $\hat{\Lambda}$.
Moreover, any representation-finite \emph{self-injective} algebra arises in this way.  In particular, a generalised Lemma \ref{lem:triv(tria)} in this setting could potentially help us classify the endomorphism algebras $\Gamma$ that are representation-finite with self-injective base.

\item As mentioned in the introduction, we expect that the condition of `$\Gamma:=\End_B(B\oplus M)$ being almost $\nu$-stale derived equivalent to $C/\soc P$' in Theorem \ref{thm:main} (3) can be strengthen to $\Gamma\cong C/\soc P$.  We give some evidence (that are not already explained previously, i.e. \eqref{eq:Triv(KQ')/soc} and Lemma \ref{EndMBTArad}) in the following.

\begin{enumerate}[(i)]
\item $B$ is of type $(\mathbb{A}_{mn},1/m)$ with $m>1$.  By our main theorem, the module $M$ must lie on the mouth of the stable AR-quiver of $B$.  Such a module is of the form $e_xB/\alpha B$ for some primitive idempotent $e_x$ corresponding to vertex $x$ on the quiver of $B$ and some arrow $\alpha$ starting from $x$; see \cite[Proposition 2.9]{CM}.  It then follows from \cite[Lemma 3.7]{CM} that $\Gamma$ is given by $C/\soc P$ with $C$ of type $(\mathbb{A}_{m(n+1)}, 1/m)$.

\item $B$ is of type $(\mathbb{A}_n,1)$ with $v(M)\in \{1,n\}$.  The argument here is completely the same as the previous case (i).

\item $B$ is of type $(\mathbb{D}_{3n},1/3)$ with $n>1$.  Then $B$ is adding an extra loop (the arrow labelled $\alpha$ in Definition \ref{def:modBSt}, or $\alpha_1$ in \cite[3.6]{Sko}) to the quiver of multiplicity-free Brauer tree (which is of RFSy-type $(\mathbb{A}_n,1)$), and adding some extra relations; see \cite{Sko}.  

Note that the unique $\tau$-orbit (which is also an $\Omega$-orbit) satisfying Theorem \ref{thm:main} for such an algebra $B$ consists of all modules of the form $e_iB/\beta B$ where $\beta$ is either an arrow in the quiver of $B$ that is not equal to the extra loop, or the maximal non-vanishing (in $B$) path for uniserial $e_iB$; this can be shown by the same argument in \cite[Proposition 2.9]{CM}.

One can then calculate explicitly $\End_B(B\oplus M)$ in a completely analogous way as \cite{CM} and see that it is isomorphic to $B'/\soc P$, where $B'$ is of type $(\mathbb{D}_{3(n+1)},1/3)$ and $P$ is the indecomposable projective corresponding to $M$.  In fact, the modified Brauer tree of $B'$ can be obtained by the same way as described in \cite{CM}, that is, adding an extra edge in the angle between $s(\beta)=i$ and $t(\beta)$ for the path $\beta$.
\end{enumerate}
\end{enumerate}


\section*{Acknowledgements}
This article started from the discussion on giving a conceptual explanation of a certain construction in the master thesis of Takuya Sakurai.  During the process of our investigation, we learnt about the work of Bernhard B\"{o}hmler \cite{B} through Ren\'{e} Marczinzik, as well as their on-going project on the same theme via a different method.  We greatly appreciate their correspondences on this matter.



\newpage
\begin{thebibliography}{AAA}
\bibitem[Asa]{Asa}
{\sc H. Asashiba},
The derived equivalence classification of representation-finite selfinjective algebras.
{\it J. Algebra} {\bf 214} (1999), no. 1, 182--221.

\bibitem[Asa2]{Asa2}
{\sc H. Asashiba},
On a lift of an individual stable equivalent to a standard derived equivalence for representation-finite self-injective algebras. 
{\it Algebr. Represent. Theory} {\bf 6} (2003), 427--447. 

\bibitem[Asa3]{Asa3}
{\sc H. Asashiba},
A covering technique for derived equivalence.
{\it J. Algebra} {\bf 191} (1997), no. 1, 382--415. 

\bibitem[AHR]{AHR}
{\sc I. Assem, D. Happel and O. Roldan},
Representation-finite trivial extension algebras.
{\it J. Pure Appl. Algebra} {\bf 33} (1984), no. 3, 235--242.

\bibitem[Aus]{Aus}
{\sc M. Auslander},
Representation theory of Artin algebras. II.
{\it Comm. Algebra} {\bf 1} (1974), 177--268.

\bibitem[Aus2]{Aus2}
{\sc M. Auslander},
Representation dimension of Artin algebras.
Math. Notes. Queen Mary College, London (1971)

\bibitem[ARS]{ARS}
{\sc M. Auslander, I. Reiten and S. O. Smalo},
Representation theory of Artin algebras. 
Cambridge Studies in Advanced Mathematics, {\bf 36}. {\it Cambridge University Press, Cambridge}, 1995. 

\bibitem[Boh]{B}
{\sc B. B\"{o}hmler},
Contributions to the representation theory of gendo-symmetric algebras.
preprint (Master thesis), 2016.

\bibitem[BLR]{BLR}
{\sc O. Bretscher, C. L\"{a}ser andC. Riedtmann},
Selfinjective and simply connected algebras, 
Manuscripta Math., {\bf 36}(1981), 253--307. 

\bibitem[Buc]{Buc}
{\sc R.-O. Buchweitz},
{Morita contexts, idempotents, and {H}ochschild cohomology--with applications to invariant rings}.
{\it Commutative algebra ({G}renoble/{L}yon, 2001), Contemp. Math.} {\bf 331}, Amer. Math. Soc., Providence, RI, (2003), 25--53.

\bibitem[CM]{CM}
{\sc A. Chan and R. Marczinzik},
On representation-finite gendo-symmetric biserial algebras.
{\it Algebr Represent. Theory} {\bf 22} (2019), no. 1, 141--176.
doi: \url{10.1007/s10468-017-9760-6}.

\bibitem[DS]{DS}
{\sc P. Dowbor and A. Skowro\'{n}ski}, 
On Galois coverings of tame algebras.
{\it Arch. Math. (Basel)} {\bf 44} (1985), no. 6, 522--529. 

\bibitem[Dro]{Dro}
{\sc Y. A. Drozd}, 
Tame and wild matrix problems.  Representations and quadratic forms,
Institute of Mathematics, Academy of Sciences Ukrainian SSR, Kiev (1979), 39--74, {\it Amer. Math. Soc. Transl.} {\bf 128} (1986), 31--55.

\bibitem[DK]{DK}
{\sc Y. A. Drozd and V. V. Kirichenko},
On quasi-Bass orders.
{\it Math. USSR-Izv.} {\bf 6} (1972), no.2, 323--365.

\bibitem[Dug]{Dug}
{\sc A. Dugas},
Tilting mutation of weakly symmetric algebras and stable equivalence. {\it Algebr. Represent. Theory} {\bf 17} (2014), no. 3, 8630--884.

\bibitem[FK1]{FK1}
{\sc M. Fang, S. Koenig},
Endomorphism algebras of generators over symmetric algebras. 
{\it J. Algebra} {\bf 332} (2011), 428--433.

\bibitem[FK2]{FK2}
{\sc M. Fang and S. Koenig},
Gendo-symmetric algebras, canonical comultiplication, bar cocomplex and dominant dimension.
{\it Trans. Amer. Math. Soc.} {\bf 368} (2016), no. 7, 5037--5055.


\bibitem[Gab]{G}
{\sc P. Gabriel},
Unzerlegbare Darstellungen. I. (German)
{\it Manuscripta Math}.{\bf 6} (1972), 71--103;
correction, ibid. {\bf 6} (1972), 309.

\bibitem[GR]{GR}
{\sc P. Gabriel and Ch. Riedtmann},
Group representations without groups.
{\it Comment. Math. Helv.} {\bf 54} (1979), no. 2, 240--287.

\bibitem[Gre]{Gre}
{\sc J.A. Green}
{\it Polynomial representatitions of $GL_n$}, Lecture Notes in Mathematics 830,
Springer-Verlag, New York, 1980.

\bibitem[Hap]{H}
{\sc D. Happel},
Triangulated categories in the representation theory of finite-dimensional algebras.
London Mathematical Society Lecture Note Series, {\bf 119}.
{\it Cambridge University Press, Cambridge}, 1998.

\bibitem[HM]{HM}
{\sc M. Hoshino and J.Miyachi},
Tame two-point algebras.
{\it Tsukuba J. Math.}{\bf 12} (1988), no. 1, 65--96.

\bibitem[HX]{HX}
{\sc W. Hu and C.C. Xi},
Derived equivalences and stable equivalences of Morita type I.
{\it Nagoya Math. J.} {\bf 200} (2010), 107--152.

\bibitem[HW]{HW}
{\sc D.Hughes and J. Waschb\"{u}sch},
Trivial extension of tilted algebras.
{\it Proc. London Math. Soc. (3)} {\bf 46}, (1983), no. 2, 347--364.

\bibitem[Kel]{K}
{\sc B. Keller},
On triangulated orbit categories.

\bibitem[Mar1]{M1}
{\sc R. Marczinzik},
Gendo-symmetric algebras, dominant dimensions and Gorenstein homological algebra.
arXiv: 1608.04212.

\bibitem[Mar2]{M2}
{\sc R. Marczinzik},
Upper bounds for dominant dimensions of gendo-symmetric algebras.
{\it Arch. Math. (Basel)} {\bf 109} (2017), no. 3, 231--243.

\bibitem[Mar3]{M3}
{\sc R. Marczinzik},
A bocs theoretic characterization of gendo-symmetric algebras.
{\it J. Algebra} {\bf 470} (2017), 160--171.

\bibitem[Mar4]{M4}
{\sc R. Marczinzik},
Finitistic Auslander algebra.
arXiv: 1701.00972

\bibitem[Mar5]{M5}
{\sc R. Marczinzik},
On weakly Gorenstein algebras.
arXiv: 1908.04738

\bibitem[Mor]{Mor}
{\sc K. Morita}, 
Duality for modules and its applications to the theory of rings with minimum condition, 
{\it Sci. Rep. Tokyo Kyoiku Daigaku Sect. A} {\bf 6} (1958) 83--142.


\bibitem[Ric]{Ric}
{\sc J. Rickard},
Derived categories and stable equivalences.
{\it J. Pure Appl. Algebra} {\bf 61} (1989), no. 3, 303--317.

\bibitem[Ric2]{Ric2}
{\sc J. Rickard},
Equivalences of derived categories for symmetric aglebras.
{\it J. Algebra} {\bf 257} (2002), 460--481.

\bibitem[Rie]{Rie}
{\sc Ch. Riedtmann}, 
Algebren, Darstellungsk\"{o}her, \"{U}berlagerungen und zur\"{u}k, {\it Comment. Math. Helv.} {\bf 55} (1980), no. 2, 199--224.

\bibitem[RZ]{RZ}
{\sc C. M. Ringel and P. Zhang},
Gorenstein-projective and semi-Gorenstein-projective modules.
arXiv: 1808.01809.

\bibitem[Sko]{Sko}
{\sc A. Skowro\'{n}ski},
Selfinjective algebras: finite and tame type.
{\it Trends in representation theory of algebras and related topics, Contemp. Math.} {\bf 406}, Amer. Math. Soc., Providence, RI (2006), 169--238.

\bibitem[Tac]{Tac}
{\sc H. Tachikawa}, 
{\it Quasi-Frobenius Rings and Generalizations}, Lecture Notes in Math., vol. 351, Springer-Verlag, Berlin, New York, 1973.

\bibitem[Was]{W}
{\sc J. Waschb\"{u}sch},
Symmetrische Algebren vom endlochen Modultyp.
{\it J. Reine Angew. Math.} {\bf 321}, (1981), 78--98.
\end{thebibliography}
\end{document}